\UseRawInputEncoding

\documentclass{amsart}

\usepackage[all,ps,cmtip]{xy}
\usepackage{amsmath, amssymb}
\usepackage{amscd}
\usepackage{tikz,color}

\newtheorem{theorem}{Theorem}[section]
\newtheorem{lemma}[theorem]{Lemma}

\newtheorem{corollary}[theorem]{Corollary}

\newtheorem{proposition}[theorem]{Proposition}
\newtheorem{conjecture}[theorem]{Conjecture}

\theoremstyle{definition}
\newtheorem{definition}[theorem]{Definition}

\theoremstyle{remark}
\newtheorem{remark}[theorem]{Remark}

\numberwithin{equation}{section}

\DeclareMathOperator{\Char}{char}

\DeclareMathOperator{\rank}{rk}

\DeclareMathOperator{\T}{T}

\DeclareMathOperator{\Hom}{Hom}

\DeclareMathOperator{\spec}{Spec}

\DeclareMathOperator{\td}{td}

\DeclareMathOperator{\Coh}{Coh}

\DeclareMathOperator{\D}{D}

\DeclareMathOperator{\K}{K}

\DeclareMathOperator{\ch}{ch}

\DeclareMathOperator{\ext}{ext}

\DeclareMathOperator{\Ext}{Ext}

\begin{document}
\title{Stability conditions on threefolds with vanishing Chern classes}

\author{Hao Max Sun}

\address{Department of Mathematics, Shanghai Normal University, Shanghai 200234, People's Republic of China}

\email{hsun@shnu.edu.cn, hsunmath@gmail.com}


\subjclass[2000]{14F17, 14F05}

\date{June 6, 2019}

\keywords{Bridgeland stability condition, Bogomolov-Gieseker
inequality, positive characteristic, Fujita's conjecture, Kodaira
vanishing theorem}

\begin{abstract}
We prove the Bogomolov-Gieseker type inequality conjectured by
Bayer, Macr\`i and Toda for threefolds with semistable tangent
bundles and vanishing Chern classes in any characteristic, which was
originally proved by Bayer, Macr\`i and Stellari in characteristic
zero. This gives the existence of Bridgeland stability conditions on
such threefolds. As applications, we obtain Reider type theorem and
confirm Fujita's conjecture for such threefolds in any
characteristic.
\end{abstract}

\maketitle

\setcounter{tocdepth}{1}

\tableofcontents

\section{Introduction}\label{S1}
Since Bridgeland's introduction in \cite{Bri1}, stability conditions
for triangulated categories have drawn a lot of attentions, and have
been investigated intensively. The existence of stability conditions
on three-dimensional varieties is often considered the biggest open
problem in the theory of Bridgeland stability conditions.

In \cite{BMT}, Bayer, Macr\`i and Toda introduced a conjectural
construction of Bridgeland stability conditions for any projective
threefold. Here the problem was reduced to proving a
Bogomolov-Gieseker type inequality for the third Chern character of
tilt-stable objects. It has been shown to hold for some Fano 3-folds
\cite{Mac2, Sch1, Li1, BMSZ, Piy}, abelian 3-folds \cite{MP1, MP2,
BMS}, \'etale quotients of abelian 3-folds \cite{BMS}, toric
threefolds \cite{BMSZ}, product threefolds of projective spaces and
abelian varieties \cite{Kos}, threefolds with nef tangent bundles
\cite{Kos2} and quintic threefolds \cite{Li2}. However,
counterexamples of the original Bogomolov-Gieseker type inequality
are found (see \cite{Sch2}). The modification of the original
inequality for any Fano threefolds is proved in \cite{BMSZ, Piy},
and it still implies the existence of stability conditions on such
threefolds. Recently, Yucheng Liu \cite{Liu} showed the existence of
stability conditions on product varieties. His method is different
from that of Bayer-Macr\`i-Toda.

In this paper, we prove the original Bogomolov-Gieseker type
inequality for threefolds with semistable tangent bundles and
vanishing Chern classes in any characteristic. This gives the
existence of Bridgeland stability conditions on such threefolds.

\begin{theorem}\label{main}
Let $X$ be a smooth projective threefold defined over an
algebraically closed field $k$, and let $H$ be an ample divisor on
$X$. Assume that $K_X\sim_{num}0$, $Hc_2(X)=0$ and $T_X$ is
$\mu_H$-semistable. Then for any $\nu_{\alpha,\beta}$-stable object
$E$ with $\nu_{\alpha,\beta}(E)=0$, we have
$$\ch_3^{\beta}(E)\leq\frac{\alpha^2}{6}H^2\ch_1^{\beta}(E).$$
\end{theorem}

By \cite[Theorem 2]{Sim} and \cite[Theorem 4.1]{Langer4}, one sees
that all the Chern classes of $X$ are vanishing under the
assumptions in Theorem \ref{main}. In characteristic zero, a well
known consequence of Yau's proof of Calabi's conjecture shows that
$X$ has a finite \'etale cover by an abelian variety if and only if
$K_X\sim_{num}0$ and $Hc_2(X)=0$. And in this case, the
semistability assumption of $T_X$ is automatically satisfied. Thus
if $\Char(k)=0$, Theorem \ref{main} is a consequence of
\cite[Theorem 1.1]{BMS} which showed the same inequality for abelian
threefolds.

In positive characteristic not much is known about the
characterizing projective varieties with vanishing Chern classes.
And there are threefolds with vanishing Chern classes which do not
have a finite \'etale cover by an abelian variety (see, e.g.,
\cite[Section 7.3]{Langer3}). Hence in some sense, Theorem
\ref{main} is new in positive characteristic. The semistable
assumption of $T_X$ in the theorem guarantees the classical
Bogomolov-Gieseker inequality to be satisfied on $X$, so that the
$\nu_{\alpha,\beta}$-stability is well defined.

The strategy of the proof is the following.  In the case of $\Char
(k)=p>0$ we compute the Euler characteristic $\chi(\mathcal{O}_X,
(F^n)^*E)$ of the pullback of $E$ by the $n$-th iteration of the
Frobenius morphism. By the Riemann-Roch theorem, one sees that
$\chi(\mathcal{O}_X, (F^n)^*E)$ is a polynomial of degree $3n$ with
respect to $p$ and its leading coefficient is $\ch_3(E)$. On the
other hand, using the tilt-stability of the Frobenius pushforward of
some locally free sheaves (see Proposition \ref{stab}), we can show
that $\ext^i(\mathcal{O}_X, (F^n)^*E)=O(p^{2n})$, for even $i$.
Taking $n\rightarrow+\infty$, we obtain an inequality for the third
Chern character of $E$. The characteristic zero case follows from
the standard spreading out technique.

\subsection*{Applications}

Theorem \ref{main} and \cite[Theorem 4.1]{BBMT} give the following
Reider type theorem:
\begin{corollary}\label{Reider}
Under the situation of Theorem \ref{main}, fix a non-negative
integer $d$. If the ample divisor $H$ satisfies
\begin{enumerate}
\item $H^3>49d$;
\item $H^2D\geq7d$ for every integral divisor class $D$ with
$H^2D>0$ and $HD^2<d$;
\item $HC\geq3d$ for any curve $C\subset X$,
\end{enumerate}
then $H^1(X, I_Z(K_X+H))=0$ for any zero-dimensional subscheme
$Z\subset X$ of length $d$. In particular, Kodaira's vanishing
theorem $H^1(X,\mathcal{O}_X(K_X+H))=0$ holds.
\end{corollary}

\begin{remark}
Theorem 4.1 in \cite{BBMT} was only showed for $d>0$ in
characteristic zero, we will prove it for $d\geq0$ in any
characteristic in Section \ref{S5}.
\end{remark}

Setting $d=1$ or $d=2$, we confirm Fujita's conjecture for such $X$
in any characteristic.
\begin{corollary}\label{Fujita}
Under the situation of Theorem \ref{main} we have
\begin{enumerate}
\item $\mathcal{O}_X(K_X+mH)$ is globally generated for $m\geq4$.
\item $\mathcal{O}_X(K_X+mH)$ is very ample for $m\geq5$.
\end{enumerate}
\end{corollary}

\begin{corollary}\label{Sheaf}
Under the situation of Theorem \ref{main}, let $c$ be the minimum
positive value of $H^2D$ for integral divisor $D$. If $Q$ is a
$\mu_H$-stable sheaf with $H^2c_1(Q)=c$, then
$$3c\ch_3(Q)\leq2(H\ch_2(Q))^2.$$
\end{corollary}
We refer to \cite[Example 4.4]{BMS} for a proof and more discussion.

In \cite{Langer3}, Langer proved that for a non-uniruled threefold
$X$ with $K_X\sim_{num}0$, the tangent bundle of $X$ is strongly
$\mu_H$-semistable for every ample divisor $H$. Hence Theorem
\ref{main}, Corollary \ref{Reider}, Corollary \ref{Fujita} and
Corollary \ref{Sheaf} hold for a non-uniruled threefold $X$ with
$K_X\sim_{num}0$ and $Hc_2(X)=0$.

\subsection*{Organization of the paper}
Our paper is organized as follows. In Section \ref{S2}, we review
basic notions and properties of some classical stabilities for
coherent sheaves, tilt-stability, the conjectural inequality
proposed in \cite{BMT, BMS}. Then in Section \ref{S3}, we show the
tilt-stability of the Frobenius pushforward of some locally free
sheaves (see Proposition \ref{stab}). Theorem \ref{main} will be
proved in Section \ref{S4}. In Section \ref{S5} we prove Corollary
\ref{Reider}.

\subsection*{Notation}
Let $X$ be a smooth projective variety defined over an algebraically
closed field $k$ of arbitrary characteristic. We denote by $T_X$ and
$\Omega_X^1$ the tangent bundle and cotangent bundle of $X$,
respectively. $K_X$ and $\omega_X$ denote the canonical divisor and
canonical sheaf of $X$, respectively. We write $c_i(X):=c_i(T_X)$
for the $i$-th Chern class of $X$, and say $X$ has vanishing Chern
classes if all the $c_i(X)$'s are numerically equivalent to zero.
Numerical equivalence of two divisors $D_1$, $D_2$ on $X$ is denoted
by $D_1\sim_{num} D_2$. For a triangulated category $\mathcal{D}$,
we write $\K(\mathcal{D})$ for the Grothendieck group of
$\mathcal{D}$.

Let $\pi:\mathcal{X}\rightarrow S$ be a flat morphism of Noetherian
schemes and $s\in S$ be a point. We denote by
$\mathcal{X}_s=\mathcal{X}\times_S\spec k(s)$ the fibre of $\pi$
over $s$, where $k(s)$ is the residue field of $s$. We write
$\mathcal{X}_{\bar{s}}=\mathcal{X}\times_S\spec \overline{k(s)}$ for
the geometric fibre of $\pi$ over $s$, here $\overline{k(s)}$ is the
algebraic closure of $k(s)$. We denote by $\D^b(\mathcal{X})$ the
bounded derived category of coherent sheaves on $\mathcal{X}$. Given
$E\in \D^b(\mathcal{X})$, we write $E_s$ (resp., $E_{\bar{s}}$) for
the pullback to the field $k(s)$ (resp., $\overline{k(s)}$).

We write $\mathcal{H}^j(E)$ ($j\in \mathbb{Z}$) for the cohomology
sheaves of a complex $E\in \D^b(X)$. We also write $H^j(F)$ ($j\in
\mathbb{Z}_{\geq0}$) for the cohomology groups of a sheaf
$F\in\Coh(X)$. Given a complex number $z\in\mathbb{C}$, we denote
its real and imaginary part by $\Re z$ and $\Im z$, respectively.
For a real number $d$, we denote by $\lceil d\rceil$ the smalleast
integer $\geq d$.


\subsection*{Convention}
Let $X$ be a smooth projective variety defined over an algebraically
closed field $k$ of characteristic $p>0$. Let
$X^{(1)}=X\times_{\spec k}\spec k$, where the product is taken over
the absolute Frobenius morphism on $\spec k$. Then the factorization
of the absolute Frobenius morphism $F : X\rightarrow X$ gives the
geometric Frobenius morphism $F_g : X\rightarrow X^{(1)}$.

The variety $X^{(1)}$ is not isomorphic to $X$ as a $k$-variety, but
$X^{(1)}$ is isomorphic to $X$ as a scheme since $F : \spec
k\rightarrow \spec k$ is an isomorphism. Hence any geometric
statement on the objects in $\D^b(X)$ is equivalent to the
corresponding statement on the objects in $\D^b(X^{(1)})$. For this
reason, we shall abuse notation and not distinguish between $X$ and
$X^{(1)}$.

\subsection*{Acknowledgments}
The author would like to thank Chunyi Li, Yucheng Liu, Xiaolei Zhao,
Zhixian Zhu, and Lei Zhang for their suggestions and comments. The
author was supported by National Natural Science Foundation of China
(Grant No. 11771294, 11301201).

\section{Preliminaries}\label{S2}
Throughout this section, we let $X$ be a smooth projective variety
of dimension $n\geq2$ defined over an algebraically closed field $k$
of arbitrary characteristic and $H$ be a fixed ample divisor on $X$.
We will review some basic notions of stability for coherent sheaves,
the weak Bridgeland stability conditions and Bogomolov-Gieseker type
inequalities.

\subsection{Stability for sheaves}
For any $\mathbb{R}$-divisor $D$ on $X$, we define the twisted Chern
character $\ch^{D}=e^{-D}\ch$. More explicitly, we have
\begin{eqnarray*}
\begin{array}{lcl}
\ch^{D}_0=\ch_0=\rank  && \ch^{D}_2=\ch_2-D\ch_1+\frac{D^2}{2}\ch_0\\
&&\\
\ch^{D}_1=\ch_1-D\ch_0 &&
\ch^{D}_3=\ch_3-D\ch_2+\frac{D^2}{2}\ch_1-\frac{D^3}{6}\ch_0.
\end{array}
\end{eqnarray*}

The first important notion of stability for a sheaf is slope
stability, also known as Mumford stability. We define the slope
$\mu_{H, D}$ of a coherent sheaf $E\in \Coh(X)$ by
\begin{eqnarray*}
\mu_{H, D}(E)= \left\{
\begin{array}{lcl}
+\infty,  & &\mbox{if}~\ch^D_0(E)=0,\\
&&\\
\frac{H^{n-1}\ch_1^{D}(E)}{H^n\ch_0^{D}(E)}, & &\mbox{otherwise}.
\end{array}\right.
\end{eqnarray*}

\begin{definition}\label{def2.1}
A coherent sheaf $E$ on $X$ is $\mu_{H, D}$-(semi)stable (or
slope-(semi)stable) if, for all non-zero subsheaves
$F\hookrightarrow E$, we have
$$\mu_{H, D}(F)<(\leq)\mu_{H, D}(E/F).$$ We say a $\mu_{H, D}$-semistable sheaf $E$ is strongly $\mu_{H,
D}$-semistable if either $\Char k = 0$ or $\Char k>0$ and all the
Frobenius pull backs of $E$ are $\mu_{H, D}$-semistable.
\end{definition}
Note that $\mu_{H, D}$ only differs from $\mu_{H}:=\mu_{H, 0}$ by a
constant, thus $\mu_{H, D}$-stability and $\mu_{H}$-stability
coincide. Harder-Narasimhan filtrations (HN-filtrations, for short)
with respect to $\mu_{H, D}$-stability exist in $\Coh(X)$: given a
non-zero sheaf $E\in\Coh(X)$, there is a filtration
$$0=E_0\subset E_1\subset\cdots\subset E_m=E$$
such that: $G_i:=E_i/E_{i-1}$ is $\mu_{H, D}$-semistable, and
$\mu_{H, D}(G_1)>\cdots>\mu_{H, D}(G_m)$. We set $\mu^+_{H,
D}(E):=\mu_{H, D}(G_1)$ and $\mu^-_{H, D}(E):=\mu_{H, D}(G_m)$.

\subsection{Weak Bridgeland stability conditions}\label{S2.2}
 The notion of ``weak
Bridgeland stability condition'' and its variant ``very weak
Bridgeland stability condition'' have been introduced in
\cite[Section 2]{Toda1} and \cite[Definition 12.1]{BMS},
respectively. We will use a slightly different notion in order to
adapt our situation. The main difference is the rotation of the
half-plane in $\mathbb{C}$.
\begin{definition}
A weak Bridgeland stability condition on $X$ is a pair $\sigma=(Z,
\mathcal{A})$, where where $\mathcal{A}$ is the heart of a bounded
$t$-structure on $\D^b(X)$, and $Z:\K(\D^b(X))\rightarrow
\mathbb{C}$ is a group homomorphism (called central charge) such
that
\begin{itemize}
\item $Z$ satisfies the following positivity property for any $E\in \mathcal{A}$:
$$Z(E)\in\{re^{i\pi\phi}: r\geq0, 0<\phi\leq1\}.$$
\item Every
non-zero object in $\mathcal{A}$ has a Harder-Narasimhan filtration
in $\mathcal{A}$ with respect to $\nu_Z$-stability, here the slope
$\nu_Z$ of an object $E\in \mathcal{A}$ is defined by
\begin{eqnarray*}
\nu_{Z}(E)= \left\{
\begin{array}{lcl}
+\infty,  & &\mbox{if}~\Im Z(E)=0,\\
&&\\
-\frac{\Re Z(E)}{\Im Z(E)}, & &\mbox{otherwise}.
\end{array}\right.
\end{eqnarray*}
\end{itemize}
\end{definition}

Let $\alpha>0$ and $\beta$ be two real numbers. We will construct a
family of weak Bridgeland stability conditions on $X$ that depends
on these two parameters. For brevity, we write $\ch^{\beta}$ for the
twisted Chern character $\ch^{\beta H}$.

There exists a \emph{torsion pair} $(\mathcal{T}_{\beta
H},\mathcal{F}_{\beta H})$ in $\Coh(X)$ defined as follows:
\begin{eqnarray*}
\mathcal{T}_{\beta H}=\{E\in\Coh(X):\mu^-_{H}(E)>\beta \}\\
\mathcal{F}_{\beta H}=\{E\in\Coh(X):\mu^+_{H}(E)\leq\beta \}
\end{eqnarray*}
Equivalently, $\mathcal{T}_{\beta H}$ and $\mathcal{F}_{\beta H}$
are the extension-closed subcategories of $\Coh(X)$ generated by
$\mu_{H, \beta H}$-stable sheaves of positive and non-positive
slope, respectively.

\begin{definition}
We let $\Coh^{\beta H}(X)\subset \D^b(X)$ be the extension-closure
$$\Coh^{\beta H}(X)=\langle\mathcal{T}_{\beta H}, \mathcal{F}_{
\beta H}[1]\rangle.$$
\end{definition}

By the general theory of torsion pairs and tilting \cite{HRS},
$\Coh^{\beta H}(X)$ is the heart of a bounded t-structure on
$\D^b(X)$; in particular, it is an abelian category. Consider the
following central charge
$$Z_{\alpha, \beta}(E)=H^{n-2}\Big(\frac{\alpha^2 H^2}{2}\ch_0^{\beta}(E)-\ch_2^{\beta}(E)+i H\ch_1^{\beta}(E)
\Big).$$ We think of it as the composition
$$Z_{\alpha, \beta}: \K(\D^b(X))\xrightarrow{\ch_H} \mathbb{Q}^3 \xrightarrow{z_{\alpha,\beta}}
\mathbb{C},$$ where the first map is given by
$$\ch_H(E)=(H^n\ch_0(E), H^{n-1}\ch_1(E), H^{n-2}\ch_2(E)),$$
and the second map is defined by
\begin{equation*}
z_{\alpha, \beta}(e_0, e_1,
e_2)=\frac{1}{2}(\alpha^2-\beta^2)e_0+\beta e_1-e_2+i(e_1-\beta
e_0).
\end{equation*}

\begin{definition}
We say $(X, H)$ satisfies Bogomolov's inequality, if
$$H^{n-2}\Delta(E):=H^{n-2}\big(\ch_1^2(E)-2\ch_0(E)\ch_2(E)\big)\geq0$$ for
any $\mu_H$-semistable sheaf $E$ on $X$.
\end{definition}

\begin{theorem}\label{thm2.5}
If $(X, H)$ satisfies Bogomolov's inequality, then for any $(\alpha,
\beta)\in \mathbb{R}_{>0}\times\mathbb{R}$, $\sigma_{\alpha,
\beta}=(Z_{\alpha, \beta}, \Coh^{\beta H}(X))$ is a weak Bridgeland
stability condition.
\end{theorem}
\begin{proof}
The required assertion is proved in \cite{Bri2, AB} for the surface
case. For the threefold case, the conclusion is showed in \cite{BMT,
BMS}. But the proof in \cite[Appendix 2]{BMS} still works for the
general case.
\end{proof}

\begin{corollary}\label{tilt}
Assume that either $\Char(k)=0$ or $T_X$ is $\mu_H$-semistable and
$K_X\sim_{num}0$. Then $(X, H)$ satisfies Bogomolov's inequality,
and for any $(\alpha, \beta)\in \mathbb{R}_{>0}\times\mathbb{R}$,
$\sigma_{\alpha, \beta}=(Z_{\alpha, \beta}, \Coh^{\beta H}(X))$ is a
weak Bridgeland stability condition.
\end{corollary}
\begin{proof}
It is well known that Bogomolov's inequality holds in characteristic
zero (see \cite[Theorem 3.4.1]{HL}). In positive characteristic
Langer \cite{Langer1} proved that the same inequality holds for
strongly $\mu_H$-semistable sheaves. Mehta and Ramanathan \cite{MR}
showed that if $X$ satisfies $\mu_H^+(\Omega_X^1)\leq0$, then all
$\mu_H$-semistable sheaves on $X$ are strongly $\mu_H$-semistable.
Thus Bogomolov's inequality holds under our assumptions.
\end{proof}

We now suppose the assumption in the above Corollary holds. We write
$\nu_{\alpha, \beta}$ for the slope function on $\Coh^{ \beta H}(X)$
induced by $Z_{\alpha, \beta}$. Explicitly, for any $E\in \Coh^{
\beta H}(X)$, one has
\begin{eqnarray*}
\nu_{\alpha, \beta}(E)= \left\{
\begin{array}{lcl}
+\infty,  & &\mbox{if}~H^{n-1}\ch^{\beta}_1(E)=0,\\
&&\\
\frac{H^{n-2}\ch_2^{\beta}(E)-\frac{1}{2}\alpha^2H^n\ch^{\beta}_0(E)}{H^{n-1}\ch^{\beta}_1(E)},
& &\mbox{otherwise}.
\end{array}\right.
\end{eqnarray*}
Corollary \ref{tilt} gives the notion of tilt-stability:

\begin{definition}
An object $E\in\Coh^{\beta H}(X)$ is \emph{tilt-(semi)stable} (or
$\nu_{\alpha,\beta}$-\emph{(semi)stable}) if, for all non-trivial
subobjects $F\hookrightarrow E$, we have
$$\nu_{\alpha, \beta}(F)<(\leq)\nu_{\alpha, \beta}(E/F).$$
\end{definition}
For any $E\in\Coh^{\beta H}(X)$, the Harder-Narasimhan property
gives a filtration in $\Coh^{\beta H}(X)$
$$0=E_0\subset E_1\subset\cdots\subset E_m=E$$ such that:
$F_i:=E_i/E_{i-1}$ is $\nu_{\alpha,\beta}$-semistable with
$\nu_{\alpha,\beta}(F_1)>\cdots>\nu_{\alpha,\beta}(F_m)$.

Tilt-stability conditions satisfy well-behaved wall-crossing:
\begin{proposition}\label{Wall}
There exists a chamber decomposition of the $(\alpha,\beta)$ half
plane by a local finite set of walls such that for any
$E\in\D^b(X)$, the Harder-Narasimhan filtration of $E$ is unchanged
in the open part of every chamber.
\end{proposition}
\begin{proof}
See \cite[Proposition 12.5]{BMS}.
\end{proof}

\subsection{Bogomolov-Gieseker type inequality}We now recall the
Bogomolov-Gieseker type inequality for tilt-stable complexes
proposed in \cite{BMT, BMS}.
\begin{definition}
We define the generalized discriminant
$$\overline{\Delta}^{\beta H}_H:=(H^{n-1}\ch^{\beta}_1)^2-2H^n\ch^{\beta}_0\cdot(H^{n-2}\ch^{\beta}_2).$$
\end{definition}
A short calculation shows $$\overline{\Delta}^{\beta
H}_H=(H^{n-1}\ch_1)^2-2H^n\ch_0\cdot(H^{n-2}\ch_2)=\overline{\Delta}_H.$$
Hence the generalized discriminant is independent of $\beta$.

\begin{theorem}\label{thm2.9}
Under the assumption in Corollary \ref{tilt}, if $E\in\Coh^{\beta
H}(X)$ is $\nu_{\alpha,\beta}$-semistable, then
$\overline{\Delta}_H(E)\geq0$.
\end{theorem}
\begin{proof}
This inequality was proved in \cite[Theorem 7.3.1]{BMT} and
\cite[Theorem 3.5]{BMS} on threefolds in characteristic zero. Since
Bogomolov's inequality holds under the assumption in Corollary
\ref{tilt}, their proof still works in our situation.
\end{proof}

\begin{conjecture}[{\cite[Conjecture 1.3.1]{BMT}}]\label{Conjecture}
Assume that $n=3$, $\Char(k)=0$ and $E\in\Coh^{\beta H}(X)$ is
$\nu_{\alpha,\beta}$-semistable with $\nu_{\alpha,\beta}(E)=0$. Then
we have
\begin{equation}\label{BG}
\ch_3^{\beta}(E)\leq\frac{\alpha^2}{6}H^2\ch_1^{\beta}(E).
\end{equation}
\end{conjecture}

Such an inequality provides a way to construct Bridgeland stability
conditions on threefolds. Recently, Schmidt \cite{Sch2} found a
counterexample to Conjecture \ref{Conjecture} when $X$ is the blowup
at a point of $\mathbb{P}^3$. Therefore, the inequality (\ref{BG})
needs some modifications in general setting. See \cite{Piy} and
\cite{BMSZ} for the recent progress.

\begin{definition}
Assume that $n=3$ and $(X,H)$ satisfies the assumption in Corollary
\ref{tilt}. For any object $E\in\Coh^{\beta H}(X)$, we define
\begin{eqnarray*}
\overline{\beta}(E)= \left\{
\begin{array}{lcl}
\frac{H^2\ch_1(E)-\sqrt{\overline{\Delta}_H(E)}}{H^3\ch_0(E)},  & &\mbox{if}~\ch_0(E)\neq0,\\
&&\\
\frac{H\ch_2(E)}{H^{2}\ch_1(E)}, & &\mbox{otherwise}.
\end{array}\right.
\end{eqnarray*}
Moreover, we say that $E$ is $\overline{\beta}$-(semi)stable, if it
is $\nu_{\alpha,\beta}$-(semi)stable in an open neighborhood of $(0,
\overline{\beta}(E))$ in $(\alpha, \beta)$-plane.
\end{definition}

Conjecture \ref{Conjecture} can be reduced as follows:
\begin{theorem}[{\cite[Theorem 5.4]{BMS}}]\label{reduction}
Assume that $n=3$, $\Char(k)=0$ and for any
$\overline{\beta}$-stable object $E\in\Coh^{\beta H}(X)$ with
$\overline{\beta}(E)\in[0, 1)$ and $\ch_0(E)\geq0$ the inequality
$$\ch_3^{\overline{\beta}(E)}(E)\leq0$$ holds. Then Conjecture
\ref{Conjecture} holds.
\end{theorem}

\section{Tilt-stability of Frobenius direct images}\label{S3}
Throughout this section, we let $k$ be an algebraically closed field
of characteristic $p>0$ and $X$ be a smooth projective variety of
dimension $n$ defined over $k$. We fix an ample divisor $H$ on $X$.
Assume that $K_X\sim_{num}0$, $H^{n-2}c_2(X)=0$ and $T_X$ is
$\mu_H$-semistable. Let $F:X\rightarrow X$ be the absolute Frobenius
morphism. We will investigate the tilt-stability of $F_*\mathcal{E}$
for a locally free sheaf $\mathcal{E}$ on $X$.

\begin{lemma}\label{Chern}
Let $\mathcal{E}$ be a locally free sheaf on $X$. Then we have
$$H^{n-i}\ch_i(F_*\mathcal{E})=p^{n-i}H^{n-i}\ch_i(\mathcal{E})$$ for
$i=0,1,2$ and
$\overline{\Delta}_H(F_*\mathcal{E})=p^{2n-2}\overline{\Delta}_H(\mathcal{E})$.
\end{lemma}
\begin{proof}
The similar computations have been done by the author in
\cite[Section 7]{Sun2}. We repeat them here for the reader's
convenience.

From the Grothendieck-Riemann-Roch theorem, it follows that
$$\ch(F_*\mathcal{E})\td(X)=F_*\big(\ch(\mathcal{E})\td(X)\big).$$
Since $\td(X)=1+\frac{1}{2}c_1+\frac{1}{12}(c_1^2+c_2)+\cdots$, the
above equation implies
\begin{eqnarray*}
\ch_0(F_*\mathcal{E})&=&p^n\ch_0(\mathcal{E})\\
\frac{1}{2}\ch_0(F_*\mathcal{E})c_1+\ch_1(F_*\mathcal{E})&=&p^{n-1}(\frac{c_1}{2}\ch_0(\mathcal{E})+c_1(\mathcal{E}))\\
\frac{c_1^2+c_2}{12}\ch_0(F_*\mathcal{E})+\frac{c_1}{2}\ch_1(F_*\mathcal{E})+\ch_2(F_*\mathcal{E})
&=&p^{n-2}(\frac{c_1^2+c_2}{12}\ch_0(\mathcal{E})+\frac{c_1}{2}c_1(\mathcal{E})+\ch_2(\mathcal{E})).
\end{eqnarray*}
By our assumptions on $c_1$ and $c_2$, a simple computation shows
$H^{n-i}\ch_i(F_*\mathcal{E})=p^{n-i}H^{n-i}\ch_i(\mathcal{E})$ for
$i=0,1,2$. Hence
$\overline{\Delta}_H(F_*\mathcal{E})=p^{2n-2}\overline{\Delta}_H(\mathcal{E})$.
\end{proof}

\begin{lemma}\label{Frob}
Let $\mathcal{E}$ be a $\mu_H$-semistale locally free sheaf on $X$.
Then $F_*\mathcal{E}$ is $\mu_H$-semistable.
\end{lemma}
\begin{proof}
Xiaotao Sun \cite{SunX1} proved that the stability of
$F_*\mathcal{E}$ depends on the stability of $\T^l(\Omega_X^1)$,
$0\leq l\leq n(p-1)$. On the other hand, by \cite[Theorem 2.1]{MR}
one sees that under our assumptions $\Omega_X^1$ and $\mathcal{E}$
are strongly $\mu_H$-semistable. So is
$\mathcal{E}\otimes\T^l(\Omega_X^1)$ for any $0\leq l\leq n(p-1)$.
From \cite[Theorem 4.8]{SunX1}, it follows that $F_*\mathcal{E}$ is
$\mu_H$-semistable.
\end{proof}

\begin{remark}
Theorem 4.8 in \cite{SunX1} was showed for the geometric Frobenius
morphism $F_g : X\rightarrow X^{(1)}$, but it still holds for the
absolute Frobenius morphism $F : X\rightarrow X$. The reason is
that, as we mentioned in Convention in Section \ref{S1}, $X^{(1)}$
is isomorphic to $X$ as a scheme. Hence the stability of
$F_*\mathcal{E}$ is equivalent to the stability of
$F_{g,*}\mathcal{E}$.
\end{remark}

\begin{proposition}\label{stab}
Let $m$ and $l$ be two integers. Let $L$ be a divisor on $X$ and
$\mathcal{G}$ be a $\nu_{\alpha,\beta}$-semistable object in an open
neighborhood of $(0,\beta_0)$ in $(\alpha,\beta)$-plane. Assume that
$L\sim_{num}mH$, $l>0$ and
$$\lim_{(\alpha,\beta)\rightarrow(0,\beta_0)}\nu_{\alpha,\beta}(\mathcal{G})=0.$$
Then
\begin{enumerate}
\item $\hom((F^l)_*\mathcal{O}_X(L), \mathcal{G})=0$
if $\beta_0<\frac{m}{p^l}$.
\item $\hom(\mathcal{G},(F^l)_*\mathcal{O}_X(L)[1])=0$
if $\beta_0>\frac{m}{p^l}$.
\end{enumerate}
\end{proposition}
\begin{proof}
By Lemma \ref{Chern} and \ref{Frob}, one sees that
$\mathcal{E}:=(F^l)_*\mathcal{O}_X(L)$ is $\mu_H$-semistable with
$$\left(H^n\ch_0(\mathcal{E}),H^{n-1}\ch_1(\mathcal{E})
,H^{n-2}\ch_2(\mathcal{E})\right)=(p^{ln}H^n,p^{l(n-1)}mH^n,\frac{1}{2}p^{l(n-2)}m^2H^n).$$
This implies
$\mu_H(\mathcal{E})=\frac{p^{l(n-1)}mH^n}{p^{ln}H^n}=\frac{m}{p^l}$
and $\overline{\Delta}_H(\mathcal{E})=0$. Consider its
Jordan-H\"older filtration
$$0=\mathcal{E}_0\subset\mathcal{E}_1\subset\cdots\subset\mathcal{E}_{s-1}\subset\mathcal{E}_s=\mathcal{E},$$
and set $\mathcal{Q}_i$ be the $\mu_H$-stable sheaf
$\mathcal{E}_{i}/\mathcal{E}_{i-1}$. It turns out that
$$\mu_H(\mathcal{Q}_i)=\mu_H(\mathcal{E}_s)=\frac{m}{p^l},$$
for any $i>0$. From \cite[Theorem 2.1]{MR}, one sees that
$\mathcal{Q}_i$ is strongly semistable. Thus by Bogomolov's
inequality for strongly semistable sheaves, we deduce that
\begin{eqnarray*}
0=\frac{\overline{\Delta}_H(\mathcal{E}_s)}{H^n\rank
\mathcal{E}_s}&=&
\mu_{H}(\mathcal{E}_s)H^{n-1}\ch_1(\mathcal{E}_s)-2H^{n-2}\ch_2(\mathcal{E}_s)\\
&=&\mu_{H}(\mathcal{E}_s)\sum_{i=1}^sH^{n-1}\ch_1(\mathcal{Q}_i)-2\sum_{i=1}^sH^{n-2}\ch_2(\mathcal{Q}_i)\\
&=&\sum_{i=1}^s\Big(\mu_{H}(\mathcal{Q}_i)H^{n-1}\ch_1(\mathcal{Q}_i)-2H^{n-2}\ch_2(\mathcal{Q}_i)\Big)\\
&=&\sum_{i=1}^s\frac{\overline{\Delta}_H(\mathcal{Q}_i)}{H^n\rank
\mathcal{Q}_i}\geq0,
\end{eqnarray*}
It follows that
$$\frac{m}{p^l}=\frac{H^{n-1}\ch_1(\mathcal{Q}_i)}{H^n\ch_0(\mathcal{Q}_i)}=\frac{2H^{n-2}\ch_2(\mathcal{Q}_i)}{H^{n-1}\ch_1(\mathcal{Q}_i)}$$
and
$$\nu_{\alpha,\beta}(\mathcal{Q}_i)=\frac{m^2-2\beta mp^l+(\beta^2-\alpha^2)p^{2l}}{2(p^lm-\beta
p^{2l})}=\frac{(m-\beta p^l)^2-\alpha^2p^{2l}}{2p^l(m-\beta p^l)}.$$
So
$$\lim_{(\alpha,\beta)\rightarrow(0,\beta_0)}\nu_{\alpha,\beta}(\mathcal{Q}_i)=\frac{1}{2}(\frac{m}{p^l}-\beta_0).$$

On the other hand, by \cite[Corollary 3.11]{BMS} or \cite[Theorem
1.3, 1.4]{Sun1} one sees that $\mathcal{Q}_i$ is
$\nu_{\alpha,\beta}$-stable for any $\alpha>0$,
$\beta<\frac{m}{p^l}$ and $\mathcal{Q}_i[1]$ is
$\nu_{\alpha,\beta}$-stable for any $\alpha>0$,
$\beta\geq\frac{m}{p^l}$. These imply that $\hom(\mathcal{Q}_i,
\mathcal{G})=0$ if $\beta_0<\frac{m}{p^l}$ and
$\hom(\mathcal{G},\mathcal{Q}_i[1])=0$ if $\beta_0>\frac{m}{p^l}$.
The conclusion of the proposition follows from
$$\hom((F^l)_*\mathcal{O}_X(L),
\mathcal{G})\leq\sum_{i=1}^s\hom(\mathcal{Q}_i, \mathcal{G})$$ and
$$\hom(\mathcal{G}, (F^l)_*\mathcal{O}_X(L)[1]
)\leq\sum_{i=1}^s\hom(\mathcal{G},\mathcal{Q}_i[1]).$$
\end{proof}

\section{The proof of the main theorem}\label{S4}
In this section, we will prove Theorem \ref{main}. The standard
spreading out technique and the Frobenius morphism will be used. We
keep the notations in Theorem \ref{main}.

In the case of $\Char(k)=0$, there is a subring $R\subset k$,
finitely generated over $\mathbb{Z}$, and a scheme
$$\pi:
\mathcal{X}\rightarrow S=\spec R$$ so that $\pi$ is smooth,
projective and $X=\mathcal{X}\times_{R}k$. We also have an object
$\mathcal{E}\in\D^b(\mathcal{X})$ and a divisor $\mathcal{H}$ on
$\mathcal{X}$ such that $E=\mathcal{E}\times_Rk$ and
$\mathcal{H}=H\times_Rk$. By the openness of semistability, one sees
that $\mathcal{X}_{s}$ satisfies the assumptions in Theorem
\ref{main} for a general point $s\in S$. Since the semistability of
sheaves is preserved by field extensions and Bogomolov's inequality
holds for any $\mu_{\mathcal{H}_{\bar{s}}}$-semistable sheaves on
the geometric fiber of $\pi$ over a general point $s\in S$,
Bogomolov's inequality holds for any
$\mu_{\mathcal{H}_{s}}$-semistable sheaves on the fiber of $\pi$
over a general point $s\in S$. By the openness of tilt-stability
(see \cite[Proposition 25.3]{BLMNPS}), it follows that for a general
closed point $s\in S$, $\mathcal{E}_s\in\Coh^{\beta
\mathcal{H}_s}(\mathcal{X}_s)$ is $\nu_{\alpha,\beta}$-stable. From
\cite[Theorem 12.17]{BLMNPS}, the same thing holds for the object
$\mathcal{E}_{\bar{s}}\in\Coh^{\beta
\mathcal{H}_{\bar{s}}}(\mathcal{X}_{\bar{s}})$. Since
$$\ch_3^{\beta}(E)-\frac{\alpha^2}{6}H^2\ch_1^{\beta}(E)=\ch_3^{\beta}(\mathcal{E}_{\bar{s}})-\frac{\alpha^2}{6}
\mathcal{H}_{\bar{s}}^2\ch_1^{\beta}(\mathcal{E}_{\bar{s}}),$$
Theorem \ref{main} will be proved if one can show it in positive
characteristic. Thus we may assume that  $\Char(k)=p>0$ and denote
by $F: X\rightarrow X$ the absolute Frobenius morphism in this
section.

By Theorem \ref{reduction}, Theorem \ref{main} will be proved, if we
can show the following:
\begin{theorem}\label{case1}
Under the situation of Theorem \ref{main}, let $E\in\Coh^{\beta
H}(X)$ be a $\overline{\beta}$-stable object with
$\overline{\beta}(E)\in[0, 1)$ and $\ch_0(E)\geq0$. Then we have
$\ch_3^{\overline{\beta}(E)}(E)\leq0$.
\end{theorem}

Since the statement of Theorem \ref{case1} is independent of scaling
$H$, we will assume throughout this section that $H$ is very ample.

\subsection{Proof of Theorem \ref{case1}, integral case} Assume that
$\overline{\beta}(E)=0$, i.e., $$H\ch_2(E)=0=K_X\ch_2(E).$$ We want
to show that $\ch_3(E)\leq0$.

We assume the contrary $\ch_3(E)>0$, and so $\ch_3(E)\geq1$. Since
$H^2\ch^{\overline{\beta}(E)}_1(E)=H^2\ch_1(E)\geq0$ and
$\ch_0(E)\geq0$, by using the Riemann-Roch theorem we can compute
\begin{eqnarray*}
\chi\big(\mathcal{O}_{X}, (F^n)^*E\big)
=p^{3n}\ch_3(E)+O(p^{2n})\geq p^{3n}+O(p^{2n}),
\end{eqnarray*}
for any positive integer $n$. On the other hand, since $E$ is a two
term complex concentrated in degree $-1$ and $0$, one sees
\begin{eqnarray*}
\chi\big(\mathcal{O}_{X}, (F^n)^*E\big)
\leq\hom\big(\mathcal{O}_{X},
(F^n)^*E\big)+\ext^2\big(\mathcal{O}_{X}, (F^n)^*E\big).
\end{eqnarray*}
Our goal is to bound from above the right hand side of this
inequality with a lower order in $p^n$.

\bigskip
\textbf{\emph{Bound on $\hom(\mathcal{O}_{X}, (F^n)^*E)$}}
\bigskip

We want to show
\begin{equation}\label{4.3}
\hom(\mathcal{O}_{X}, (F^n)^*E)=O(p^{2n}).
\end{equation}
We consider the exact triangle in $\D^b(X)$
$$(F^n)^*E\otimes\mathcal{O}_X(-H)
\rightarrow (F^n)^*E \rightarrow((F^n)^*E)\otimes\mathcal{O}_{Y},$$
where $Y$ is a general smooth surface in $|H|$. It follows that
\begin{eqnarray*}
\hom(\mathcal{O}_{X}, (F^n)^*E)\leq\hom(\mathcal{O}_{X},
(F^n)^*E\otimes\mathcal{O}_X(-H))+\hom(\mathcal{O}_{X},
((F^n)^*E)\otimes\mathcal{O}_{Y}).
\end{eqnarray*}
By Serre duality and adjointness between $(F^n)^*$ and $(F^n)_*$,
one obtains
\begin{eqnarray*}
\hom(\mathcal{O}_{X},
(F^n)^*E\otimes\mathcal{O}_X(-H))=\hom((F^n)_*\mathcal{O}_{X}(H+K_X),
E\otimes\omega_X).
\end{eqnarray*}
Since $K_X\sim_{num}0$, Proposition \ref{stab} gives
$\hom((F^n)_*\mathcal{O}_{X}(H+K_X), E\otimes\omega_X)=0$. Thus we
have
\begin{eqnarray*}
\hom(\mathcal{O}_{X}, (F^n)^*E)\leq\hom(\mathcal{O}_{X},
((F^n)^*E)\otimes\mathcal{O}_{Y}).
\end{eqnarray*}

We then consider the cohomology sheaves of $E$ and the exact
triangle in $\D^b(X)$
$$\mathcal{H}^{-1}(E)[1]\rightarrow E\rightarrow\mathcal{H}^{0}(E).$$
Since $Y$ is general, \cite[Lemma 7.1]{BMS} gives
\begin{eqnarray*}
\hom(\mathcal{O}_{X}, ((F^n)^*E)\otimes\mathcal{O}_{Y})\leq
h^0\big(\left((F^n)^*\mathcal{H}^0(E)\right)|_{Y}\big)+h^1\big(\left((F^n)^*\mathcal{H}^{-1}(E)\right)|_{Y}\big).
\end{eqnarray*}
The bound (\ref{4.3}) will then follow from the following lemma.

\begin{lemma}\label{est}
Let $\mathcal{Q}$ be a sheaf and $\mathcal{L}$ be a line bundle on
$X$. Let $Y$ be a general smooth surface in the very ample linear
system $|bH|$, where $b$ is a positive integer. Then for any $0\leq
i\leq2$, there are rational numbers $a_i$ ($1\leq i\leq 6$) which
are independent of $n$ and $\mathcal{L}$ such that
\begin{eqnarray*}
h^i(Y,((F^n)^*\mathcal{Q}\otimes \mathcal{L})|_{Y})\leq
a_1p^{2n}+a_2\mu_H(\mathcal{L})p^n+a_3p^n+a_4\mu_H(\mathcal{L})^2+a_5\mu_H(\mathcal{L})+a_6.
\end{eqnarray*}
\end{lemma}
\begin{proof}
We denote by $F_Y$ the absolute Frobenius morphism of $Y$ and assume
first that $\mathcal{Q}$ is torsion free. Take a positive integer
$a$ such that $T_Y(aH|_Y)$ is globally generated. Since
$\big((F^n)^*\mathcal{Q}\big)|_{Y}=(F_Y^n)^*(\mathcal{Q}|_{Y})$, by
\cite[Corollary 2.5]{Langer1}, one obtains that
\begin{eqnarray*}
\mu^+_{H|_{Y}}\big(((F^n)^*\mathcal{Q})|_{Y}\big)&\leq&
p^{n}\mu^+_{H|_{Y}}\big(\mathcal{Q}|_{Y}\big)+\frac{p^n(\rank\mathcal{Q}-1)
}{p-1}abH^3\\
\mu^-_{H|_{Y}}\big(((F^n)^*\mathcal{Q})|_{Y}\big)&\geq&
p^{n}\mu^-_{H|_{Y}}\big(\mathcal{Q}|_{Y}\big)-\frac{p^n(\rank\mathcal{Q}-1)
}{p-1}abH^3.
\end{eqnarray*}
Hence
\begin{eqnarray*}
\mu^+:=\mu^+_{H_Y}\big(((F^n)^*\mathcal{Q}\otimes
\mathcal{L})|_{Y}\big)&=&\mu^+_{H|_{Y}}\big(((F^n)^*\mathcal{Q})|_{Y}\big)+\mu_{H|_{Y}}(\mathcal{L}|_Y)\\
&\leq&
p^{n}\mu^+_{H|_{Y}}\big(\mathcal{Q}|_{Y}\big)+\frac{p^n(\rank\mathcal{Q}-1)
}{p-1}abH^3+\mu_{H|_{Y}}(\mathcal{L}|_Y)\\
&=&p^{n}\mu^+_{H|_{Y}}\big(\mathcal{Q}|_{Y}\big)+\frac{p^n(\rank\mathcal{Q}-1)
}{p-1}abH^3+\mu_{H}(\mathcal{L})
\end{eqnarray*}
From Langer's estimation \cite[Theorem 3.3]{Langer2}, it follows
that
\begin{eqnarray*}
&&h^0(Y,((F^n)^*\mathcal{Q}\otimes \mathcal{L})|_{Y})\\
&\leq& \left\{
\begin{array}{lcl}
\frac{(\rank \mathcal{Q})bH^3}{2}\Big(\mu^++f(\rank
\mathcal{Q})+2\Big)\Big(\mu^++f(\rank \mathcal{Q})+1\Big),  & &\mbox{if}~\mu^+\geq0\\
&&\\
0, & &\mbox{otherwise}
\end{array}\right.\\
&\leq&b_1p^{2n}+b_2\mu_H(\mathcal{L})p^n+b_3p^n+b_4\mu_H(\mathcal{L})^2+b_5\mu_H(\mathcal{L})+b_6,
\end{eqnarray*}
where $f(\rank \mathcal{Q})=-1+\sum_{i=1}^{\rank
\mathcal{Q}}\frac{1}{i}$ and $b_i$'s are independent of $n$ and
$\mathcal{L}$.

The $h^2$-estimate follows similarly, by using Serre Duality. For
$h^1$, the Riemann-Roch theorem gives
\begin{eqnarray*}
h^1(Y,((F^n)^*\mathcal{Q}\otimes
\mathcal{L})|_{Y})&=&h^0(Y,((F^n)^*\mathcal{Q}\otimes
\mathcal{L})|_{Y})+h^2(Y,((F^n)^*\mathcal{Q}\otimes
\mathcal{L})|_{Y})\\&&-\chi(Y,((F^n)^*\mathcal{Q}\otimes
\mathcal{L})|_{Y}).
\end{eqnarray*}
It follows that the upper bound of $h^1$ has the same form as that
of $h^0$. This finishes the proof in the torsion-free case. The
proof for a general sheaf $\mathcal{Q}$ is the same as that of
\cite[Lemma 7.3]{BMS}.
\end{proof}

\bigskip
\textbf{\emph{Bound on $\ext^2\big(\mathcal{O}_{X}, (F^n)^*E\big)$}}
\bigskip

This is similar to the previous case. We consider the exact triangle
$$(F^n)^*E
\rightarrow ((F^n)^*E)\otimes\mathcal{O}_{X}(H) \rightarrow
\left((F^n)^*E\right)\otimes\mathcal{O}_{Y}(H).$$ By Proposition
\ref{stab}, Serre duality and the adjointness, one obtains
\begin{eqnarray*}
\ext^2\big(\mathcal{O}_{X}, ((F^n)^*E)\otimes\mathcal{O}_{X}(H)\big)
&=&\ext^1\big((F^n)^*E, \omega_{X}(-H)
\big)\\
&=&\ext^1\big(E,
(F^n)_*(\omega_{X}(-H))\big)\\
&=&\hom\big(E,
(F^n)_*(\omega_{X}(-H))[1]\big)\\
&=&0.
\end{eqnarray*}
Thus Lemma \ref{est} gives
\begin{eqnarray*}
\ext^2\big(\mathcal{O}_{X}, (F^n)^*E\big)
&\leq&\ext^1\big(\mathcal{O}_{X},
\left((F^n)^*E\right)\otimes\mathcal{O}_{Y}(H)\big)\\
&\leq&h^1\big((F^n)^*\mathcal{H}^0(E)
\otimes\mathcal{O}_{Y}(H)\big)\\
&&+h^2\big((F^n)^*\mathcal{H}^{-1}(E)
\otimes\mathcal{O}_{Y}(H)\big)\\
&=&O(p^{2n}).
\end{eqnarray*}

In conclusion, we have
$$p^{3n}+O(p^{2n})\leq\chi\big(\mathcal{O}_{X},
(F^n)^*E\big)\leq O(p^{2n}),$$ which gives the required
contradiction for $n$ sufficiently large.

\subsection{Proof of Theorem \ref{case1}, rational case} We assume
that $\overline{\beta}(E)\in\mathbb{Q}\setminus\mathbb{Z}$ and write
$\overline{\beta}(E)=\frac{v}{p^ru}$ with $p$ and $u$ coprime and
$p^ru>v>0$. By Euler's theorem, we have $$p^{n\varphi(u)}\equiv
1\mod u$$ for any positive integer $n$, where $\varphi(u)$ is
Euler's totient function. This implies that
$c_n:=\frac{p^{n\varphi(u)}-1}{u}$ is an integer and
\begin{equation}\label{4.2}
\frac{c_nv}{p^{n\varphi(u)+r}}=(1-\frac{1}{p^{n\varphi(u)}})\overline{\beta}(E).
\end{equation}

Set $a_n=n\varphi(u)+r$. By using the Riemann-Roch theorem we can
compute
\begin{eqnarray*}
\chi\big(\mathcal{O}_{X},
(F^{a_n})^*E\otimes\mathcal{O}_X(-c_nvH)\big)
&=&\ch_3\Big((F^{a_n})^*E\otimes\mathcal{O}_X(-c_nvH))\Big)+O(p^{2a_n})\\
&=&p^{3a_n}\Big(\ch_3^{c_nv/p^{a_n}}(E)\Big)+O(p^{2a_n}).
\end{eqnarray*}
From (\ref{4.2}), one obtains that
\begin{eqnarray*}
\ch_3^{c_nv/p^{a_n}}(E)&=&
\ch_3^{(1-\frac{1}{p^{n\varphi(u)}})\overline{\beta}(E)}(E)\\
&=&\ch_3^{\overline{\beta}(E)}(E)+(\frac{\overline{\beta}(E)}{p^{n\varphi(u)}})^2\frac{H^2\ch_1^{\overline{\beta}(E)}(E)}{2}
+(\frac{\overline{\beta}(E)}{p^{n\varphi(u)}})^3\frac{H^3\ch_0^{\overline{\beta}(E)}(E)}{6}.
\end{eqnarray*}
Hence we deduce that
\begin{eqnarray*}
\chi\big(\mathcal{O}_{X},
(F^{a_n})^*E\otimes\mathcal{O}_X(-c_nvH)\big)
&=&p^{3a_n}\Big(\ch_3^{c_nv/p^{a_n}}(E)\Big)+O(p^{2a_n})\\
&=&p^{3a_n}\Big(\ch_3^{\overline{\beta}(E)}(E)\Big)+O(p^{2a_n}).
\end{eqnarray*}
and
\begin{eqnarray*}
\chi\big(\mathcal{O}_{X},
(F^{a_n})^*E\otimes\mathcal{O}_X(-c_nvH)\big)
&\leq&\hom\big(\mathcal{O}_{X},
(F^{a_n})^*E\otimes\mathcal{O}_X(-c_nvH)\big)\\
&&+\ext^2\big(\mathcal{O}_{X},
(F^{a_n})^*E\otimes\mathcal{O}_X(-c_nvH)\big).
\end{eqnarray*}
Since
$$\frac{c_nv+p^r}{p^{a_n}}=\overline{\beta}(E)+(1-\overline{\beta}(E))\frac{1}{p^{n\varphi(u)}}>\overline{\beta}(E),$$
from Proposition \ref{stab}, it follows that
\begin{eqnarray*}
&&\hom\big(\mathcal{O}_{X},
(F^{a_n})^*E\otimes\mathcal{O}_X(-c_nvH-p^rH)\big)\\
&=&\hom\big((F^{a_n})_*\mathcal{O}_X(K_X+c_nvH+p^rH),
E\otimes\omega_X\big)\\
&=&0
\end{eqnarray*}
and
\begin{eqnarray*}
\ext^2\big(\mathcal{O}_{X},
(F^{a_n})^*E\otimes\mathcal{O}_X(-c_nvH)\big)&=&\ext^1\big(E,
(F^{a_n})_*\mathcal{O}_X(K_X+c_nvH)
\big)\\
&=&0
\end{eqnarray*}
Similar to the proof of (\ref{4.3}), one obtains
\begin{eqnarray*}
\hom\big(\mathcal{O}_{X},
(F^{a_n})^*E\otimes\mathcal{O}_X(-c_nvH)\big)
&\leq&\hom\big(\mathcal{O}_{X},
(F^{a_n})^*E\otimes\mathcal{O}_X(-c_nvH-p^rH)\big)\\
&&+\hom\big(\mathcal{O}_{X},
(F^{a_n})^*E\otimes\mathcal{O}_Z(-c_nvH)\big)\\
&=&\hom\big(\mathcal{O}_{X},
(F^{a_n})^*E\otimes\mathcal{O}_Z(-c_nvH)\big)\\
&\leq&h^0\big((F^{a_n})^*\mathcal{H}^0(E)\otimes\mathcal{O}_Z(-c_nvH)\big)\\
&&+h^1\big((F^{a_n})^*\mathcal{H}^{-1}(E)\otimes\mathcal{O}_Z(-c_nvH)\big)\\
&=&O(p^{2a_n}),
\end{eqnarray*}
where $Z$ is a general smooth surface in $|p^rH|$.

In conclusion, we have
$$p^{3a_n}\ch_3^{\overline{\beta}(E)}(E)+O(p^{2a_n})\leq\chi\big(\mathcal{O}_{X},
(F^{a_n})^*E\otimes\mathcal{O}_X(-c_nvH)\big)\leq O(p^{2a_n}).$$
This gives $\ch_3^{\overline{\beta}(E)}(E)\leq0$ by taking
$n\rightarrow+\infty$.

\subsection{Proof of Theorem \ref{case1}, irrational case} We now assume
that $\overline{\beta}(E)\in\mathbb{R}\setminus\mathbb{Q}$. By
assumption, there exists $0<\varepsilon<\overline{\beta}(E)$ such
that $E$ is $\nu_{\alpha,\beta}$-stable for all $(\alpha,\beta)$ in
$$V_{\varepsilon}:=\{(\alpha,\beta)\in\mathbb{R}_{>0}\times\mathbb{R}:0<\alpha<\varepsilon,
\overline{\beta}(E)-\varepsilon<\beta<\overline{\beta}(E)+\varepsilon\}.$$
By the Dirichlet approximation theorem, there exists a sequence
$\{\beta_n=\frac{v_n}{p^{r_n}u_n}\}_{n\in \mathbb{N}}$ of rational
numbers with $u_n>0$, $v_n>0$, $r_n\geq0$, $u_n$ and $p$ coprime and
$p^{r_n}u_n\rightarrow+\infty$ as $n\rightarrow+\infty$ such that
$$\Big|\overline{\beta}(E)-\frac{v_n}{p^{r_n}u_n}\Big|<\frac{1}{p^{2r_n}u^2_n}<\varepsilon$$
for all $n$. As in the rational case, by Euler's theorem, for any
$m\geq1$,
$$c_{mn}:=\frac{p^{m\varphi(u_n)}-1}{u_n}$$ is a positive integer. It
turns out that
\begin{eqnarray}\label{4.10}
\nonumber(1-\frac{1}{p^{m\varphi(u_n)}})(\overline{\beta}(E)-\frac{1}{p^{2r_n}u_n^2})&<&\frac{c_{mn}v_n}{p^{m\varphi(u_n)+r_n}}\\
&=&(1-\frac{1}{p^{m\varphi(u_n)}})
\beta_n\\
\nonumber&<&(1-\frac{1}{p^{m\varphi(u_n)}})(\overline{\beta}(E)+\frac{1}{p^{2r_n}u_n^2}).
\end{eqnarray}
Let $a_{mn}:=m\varphi(u_n)+r_n$ and
$Q_{mn}:=(F^{a_{mn}})^*E\otimes\mathcal{O}_X(-c_{mn}v_nH)$. We
compute, for $m\gg0$,
\begin{eqnarray}\label{RR}
&&\chi\big(\mathcal{O}_{X},
Q_{mn}\big)\\
\nonumber&=&\ch_3((F^{a_{mn}})^*E\otimes\mathcal{O}_X(-c_{mn}v_nH))+O(p^{2a_{mn}})\\
\nonumber&=&p^{3a_{mn}}\ch_3^{c_{mn}v_n/p^{a_{mn}}}(E)+O(p^{2a_{mn}})\\
\nonumber&=&p^{3a_{mn}}\Big(\ch_3^{\beta_n}(E)+\frac{\beta_n}{p^{m\varphi(u_n)}}H\ch_2^{\beta_n}(E)\\
\nonumber&&+(\frac{\beta_n}{p^{m\varphi(u_n)}})^2\frac{H^2\ch_1^{\beta_n}(E)}{2}
+(\frac{\beta_n}{p^{m\varphi(u_n)}})^3\frac{H^3\ch_0^{\beta_n}(E)}{6}\Big)+O(p^{2a_{mn}})\\
\nonumber&\geq&p^{3a_{mn}}\ch_3^{\beta_n}(E)+O(p^{2a_{mn}})\\
\nonumber&\geq&p^{3a_{mn}}\ch_3^{\overline{\beta}(E)}(E)+O(p^{2a_{mn}}).
\end{eqnarray}
The last inequality follows since, by definition, $\ch_3^{\beta}(E)$
has a local minimum at $\beta=\overline{\beta}(E)$. As in the
previous case, we want to bound
\begin{eqnarray}\label{Euler}
\chi\big(\mathcal{O}_{X}, Q_{mn}\big)\leq \hom\big(\mathcal{O}_{X},
Q_{mn}\big)+\ext^2\big(\mathcal{O}_{X}, Q_{mn}\big)
\end{eqnarray}
for $m\gg0$ and $n\gg0$.

We let
$l_0=\lceil\frac{p^{m\varphi(u_n)}-1}{p^{r_n}u_n^2}+p^{r_n}\overline{\beta}(E)\rceil$
and
$l_1=\lceil\frac{p^{m\varphi(u_n)}-1}{p^{r_n}u_n^2}-p^{r_n}\overline{\beta}(E)\rceil$.
Then by (\ref{4.10}) one has
\begin{eqnarray*}
\frac{c_{mn}v_n+l_0}{p^{a_{mn}}}&>&(1-\frac{1}{p^{m\varphi(u_n)}})(\overline{\beta}(E)-\frac{1}{p^{2r_n}u_n^2})+\frac{l_0}{p^{a_{mn}}}\\
&=&\overline{\beta}(E)-\frac{1}{p^{2r_n}u_n^2}-\frac{\overline{\beta}(E)}{p^{m\varphi(u_n)}}+\frac{1}{p^{m\varphi(u_n)+2r_n}u_n^2}+\frac{l_0}{p^{a_{mn}}}\\
&=&\overline{\beta}(E)+\frac{1}{p^{a_{mn}}}\Big(l_0-\frac{p^{m\varphi(u_n)}-1}{p^{r_n}u_n^2}-p^{r_n}\overline{\beta}(E)\Big)\\
&>&\overline{\beta}(E)
\end{eqnarray*}
and
\begin{eqnarray*}
\frac{c_{mn}v_n-l_1}{p^{a_{mn}}}&<&(1-\frac{1}{p^{m\varphi(u_n)}})(\overline{\beta}(E)+\frac{1}{p^{2r_n}u_n^2})-\frac{l_1}{p^{a_{mn}}}\\
&=&\overline{\beta}(E)+\frac{1}{p^{2r_n}u_n^2}-\frac{\overline{\beta}(E)}{p^{m\varphi(u_n)}}-\frac{1}{p^{m\varphi(u_n)+2r_n}u_n^2}-\frac{l_1}{p^{a_{mn}}}\\
&=&\overline{\beta}(E)-\frac{1}{p^{a_{mn}}}\Big(l_1-\frac{p^{m\varphi(u_n)}-1}{p^{r_n}u_n^2}+p^{r_n}\overline{\beta}(E)\Big)\\
&<&\overline{\beta}(E)
\end{eqnarray*}
Thus Proposition \ref{stab} gives
\begin{eqnarray}\label{4.5}
&&\hom(\mathcal{O}_X, Q_{mn}(-l_0H))\\
\nonumber&=&\hom\Big(\mathcal{O}_X,
(F^{a_{mn}})^*E\otimes\mathcal{O}_X(-c_{mn}v_nH-l_0H)\Big)\\
\nonumber&=&\hom\Big(
(F^{a_{mn}})_*\mathcal{O}_X(K_X+c_{mn}v_nH+l_0H),
E\otimes\omega_X\Big)\\
\nonumber&=&0
\end{eqnarray}
and
\begin{eqnarray}\label{4.6}
&&\ext^2(\mathcal{O}_X, Q_{mn}(l_1H))\\
\nonumber&=&\ext^2\Big(\mathcal{O}_X,
(F^{a_{mn}})^*E\otimes\mathcal{O}_X(-c_{mn}v_nH+l_1H)\Big)\\
\nonumber&=&\hom\Big(E,
(F^{a_{mn}})_*\mathcal{O}_X(K_X+c_{mn}v_nH-l_1H)[1]
\Big)\\
\nonumber&=&0
\end{eqnarray}

Consider the exact triangle in $\D^b(X)$
$$Q_{mn}(-(j+1)H)
\rightarrow Q_{mn}(-jH) \rightarrow
Q_{mn}(-jH)\otimes\mathcal{O}_{Y},$$ where $0\leq j\leq l_0-1$ and
$Y$ is a general smooth surface in $|H|$. From (\ref{4.5}), it
follows that
\begin{eqnarray*}
&&\hom\big(\mathcal{O}_{X},
Q_{mn}\big)\\
&\leq&\hom\big(\mathcal{O}_{X},
Q_{mn}(-l_0H)\big)+\sum_{j=0}^{l_0-1}\hom\big(\mathcal{O}_{X},
Q_{mn}(-jH)\otimes\mathcal{O}_{Y}\big)\\
&=&\sum_{j=0}^{l_0-1}\hom\big(\mathcal{O}_{X},
Q_{mn}(-jH)\otimes\mathcal{O}_{Y}\big).
\end{eqnarray*}
On the other hand, by Lemma \ref{est} and the definition of
$c_{mn}$, one sees for $m\gg0$,
\begin{eqnarray*}
&&\sum_{j=0}^{l_0-1}\hom\big(\mathcal{O}_{X},
Q_{mn}(-jH)\otimes\mathcal{O}_{Y}\big)\\
&\leq&\sum_{j=0}^{l_0-1}\big(b_1p^{2a_{mn}}+(b_2p^{a_{mn}}+b_3)(c_{mn}v_n+j)+b_4p^{a_{mn}}+b_5(c_{mn}v_n+j)^2+b_6\big)\\
&=&\sum_{j=0}^{l_0-1}\Big(b_1p^{2a_{mn}}+b_2(c_{mn}v_n+j)p^{a_{mn}}+b_5(c_{mn}v_n+j)^2\Big)+O(p^{2a_{mn}})\\
&=&l_0(b_1p^{2a_{mn}}+b_2c_{mn}v_np^{a_{mn}}+b_5c_{mn}^2v_n^2)+\frac{l_0(l_0-1)}{2}(b_2p^{a_{mn}}+2b_5c_{mn}v_n)\\
&&+\frac{b_5}{6}l_0(l_0-1)(2l_0-1)+O(p^{2a_{mn}})\\
&=&\frac{p^{a_{mn}}}{p^{2r_n}u_n^2}(b_1p^{2a_{mn}}+b_2\beta_np^{2a_{mn}}+b_5\beta_n^2p^{2a_{mn}})+\frac{p^{2a_{mn}}}{2p^{4r_n}u_n^4}(b_2p^{a_{mn}}+2b_5\beta_np^{a_{mn}})\\
&&+\frac{b_5}{3}\frac{p^{3a_{mn}}}{p^{6r_n}u_n^6}+O(p^{2a_{mn}})\\
&\leq&\Big(\frac{d_1}{p^{2r_n}u_n^2}+\frac{d_2}{p^{4r_n}u_n^4}+\frac{d_3}{p^{6r_n}u_n^6}\Big)p^{3a_{mn}}+O(p^{2a_{mn}}),
\end{eqnarray*}
where $b_i$'s and $d_j$'s are independent of $m$ and $n$. Therefore
for $m\gg0$ we have
\begin{equation}\label{4.7}
\hom\big(\mathcal{O}_{X},
Q_{mn}\big)\leq\Big(\frac{d_1}{p^{2r_n}u_n^2}+\frac{d_2}{p^{4r_n}u_n^4}+\frac{d_3}{p^{6r_n}u_n^6}\Big)p^{3a_{mn}}+O(p^{2a_{mn}}).
\end{equation}

To bound $\ext^2\big(\mathcal{O}_{X}, Q_{mn}\big)$, as before, we
consider the exact triangle in $\D^b(X)$
$$Q_{mn}((j-1)H)
\rightarrow Q_{mn}(jH) \rightarrow
Q_{mn}(jH)\otimes\mathcal{O}_{Y},$$ where $1\leq j\leq l_1$. From
(\ref{4.6}), it follows that
\begin{eqnarray*}
\ext^2(\mathcal{O}_{X}, Q_{mn})&\leq&\ext^2(\mathcal{O}_{X},
Q_{mn}(l_1H))+\sum_{j=1}^{l_1}\ext^1(\mathcal{O}_{X},
Q_{mn}(jH)\otimes\mathcal{O}_{Y})\\
&=&\sum_{j=1}^{l_1}\ext^1(\mathcal{O}_{X},
Q_{mn}(jH)\otimes\mathcal{O}_{Y}).
\end{eqnarray*}
As the same proof of (\ref{4.7}), for $m\gg0$ one obtains,
\begin{equation}\label{4.8}
\ext^2(\mathcal{O}_{X}, Q_{mn})
\leq\Big(\frac{e_1}{p^{2r_n}u_n^2}+\frac{e_2}{p^{4r_n}u_n^4}+\frac{e_3}{p^{6r_n}u_n^6}\Big)p^{3a_{mn}}+O(p^{2a_{mn}}),
\end{equation}
where the constants $e_i$'s are independent of $m$ and $n$.

In conclusion, by (\ref{RR}), (\ref{Euler}), (\ref{4.7}) and
(\ref{4.8}), we obtain, for $m\gg0$,
\begin{eqnarray*}
&&\Big(\frac{d_1+e_1}{p^{2r_n}u_n^2}+\frac{d_2+e_2}{p^{4r_n}u_n^4}+\frac{d_3+e_3}{p^{6r_n}u_n^6}\Big)p^{3a_{mn}}+O(p^{2a_{mn}})\\
&\geq&\chi\big(\mathcal{O}_{X},
Q_{mn}\big)\\
&\geq&p^{3a_{mn}}\ch_3^{\overline{\beta}(E)}(E)+O(p^{2a_{mn}}).
\end{eqnarray*}
This implies
\begin{eqnarray*}
\ch_3^{\overline{\beta}(E)}(E)\leq\frac{d_1+e_1}{p^{2r_n}u_n^2}+\frac{d_2+e_2}{p^{4r_n}u_n^4}+\frac{d_3+e_3}{p^{6r_n}u_n^6}.
\end{eqnarray*}
Taking $n\rightarrow+\infty$, we conclude that
$\ch_3^{\overline{\beta}(E)}(E)\leq0$. This completes the proof of
Theorem \ref{case1}.

\section{The proof of Corollary \ref{Reider}}\label{S5}
In this section, we will prove Corollary \ref{Reider}. It was proved
in \cite[Theorem4.1]{BBMT} in characteristic zero. The
characteristic zero assumption was only used to guarantee the
Kodaira vanishing:
$$H^1(X, \mathcal{O}_X(K_X+H))=0,$$ so that one can proceed by
induction on the length $d$ of $Z$ (see \cite[Assumption
(*)]{BBMT}). Hence Corollary \ref{Reider} holds if one can show the
following:
\begin{theorem}\label{Kodaira}
Let $X$ be a smooth projective threefold defined over an
algebraically closed field $k$, and let $H$ be an ample divisor on
$X$. Assume that Bogomolov's inequality and Conjecture
\ref{Conjecture} holds for $(X,H)$. Then we have $$H^1(X,
\mathcal{O}_X(K_X+H))=0.$$
\end{theorem}

We follow the method in \cite{AB11, BBMT} to prove the above
theorem, but avoid to use the dualizing functor. One observes that
if $H^1(X, \mathcal{O}_X(K_X+H))\neq0$, then by Serre duality, we
have
$$\Ext^2(\mathcal{O}_X(H),\mathcal{O}_X)=\Ext^1(\mathcal{O}_X(H),\mathcal{O}_X[1])\neq0.$$
Take a non-zero element
$\xi\in\Ext^1(\mathcal{O}_X(H),\mathcal{O}_X[1])$. It gives a
non-trivial exact sequence in $\Coh^{\frac{1}{2}H}(X)$:
\begin{equation}\label{5.1}
0\rightarrow\mathcal{O}_X[1]\xrightarrow{f}
E_{\xi}\rightarrow\mathcal{O}_X(H)\rightarrow0.
\end{equation}
We will study the $\nu_{\alpha,\beta}$-stability of $E_{\xi}$ for
$\alpha>0$ and $\beta=\frac{1}{2}$.

\begin{lemma}\label{lemma5.2}
The object $E_{\xi}\in\Coh^{\frac{1}{2}H}(X)$ satisfies the
following:
\begin{enumerate}
\item $\ch^{\frac{1}{2}}(E_{\xi})=(0, H, 0, \frac{1}{24}H^3)$.
\item If $\alpha>\frac{1}{2}$, then (\ref{5.1}) destabilizes
$E_{\xi}$ with respect to $\nu_{\alpha,\frac{1}{2}}$.
\item If $\alpha=\frac{1}{2}$, then $E_{\xi}$ is
$\nu_{\alpha,\frac{1}{2}}$-semistable.
\item $E_{\xi}$ is not
$\nu_{\alpha,\frac{1}{2}}$-semistable for $0<\alpha<\frac{1}{2}$.
\end{enumerate}
\end{lemma}
\begin{proof}
See \cite[Proposition 3.1]{BBMT}.
\end{proof}

By Lemma \ref{lemma5.2} and Proposition \ref{Wall}, there exists an
exact sequence in $\Coh^{\frac{1}{2}H}(X)$
$$0\rightarrow A\rightarrow E_{\xi}\rightarrow F\rightarrow0$$
with the following properties:
\begin{itemize}
\item $A$ is $\nu_{\frac{1}{2},\frac{1}{2}}$-semistable with $\nu_{\frac{1}{2},\frac{1}{2}}(A)=0$;
\item $\nu_{\alpha,\frac{1}{2}}(A)>0$ if $\alpha<\frac{1}{2}$.
\end{itemize}

\begin{proposition}\label{prop}
The object $A$ is of the form $I_Z(H)$ for some zero-dimensional
subscheme $Z\subset X$.
\end{proposition}

\begin{proof}
\noindent{\bf Step 1.} {\em $\ch_0(A)=1$. }
\medskip

The properties above imply that
$$H\ch_2^{\frac{1}{2}}(A)=\frac{1}{8}H^3\ch_0(A)$$ and
$$H\ch_2^{\frac{1}{2}}(A)-\frac{\alpha^2}{2}H^3\ch_0(A)=\frac{1}{8}H^3\ch_0(A)-\frac{\alpha^2}{2}H^3\ch_0(A)>0$$ if
$\alpha<\frac{1}{2}$. Hence $\ch_0(A)>0$. Expanding
$\ch_2^{\frac{1}{2}}(A)$, one obtains

\begin{equation}\label{5}
H\ch_2(A)=\frac{1}{2}H^2\ch_1(A).
\end{equation}
Applying Theorem \ref{thm2.9} to $A$, one has
$$\frac{1}{2}H^2\ch_1(A)=H\ch_2(A)\leq\frac{(H^2\ch_1(A))^2}{2H^3\ch_0(A)}.$$
Hence
\begin{equation}\label{5.2}
H^2\ch_1(A)\geq H^3\ch_0(A).
\end{equation}

On the other hand, since $\nu_{\alpha,\frac{1}{2}}(F)\neq+\infty$,
we deduce that
$$0<H^2\ch^{\frac{1}{2}}_1(F)=H^2\ch^{\frac{1}{2}}_1(E_{\xi})-H^2\ch^{\frac{1}{2}}_1(A).$$
It follows that
\begin{equation}\label{5.3}
H^2\ch_1(A)-\frac{1}{2}H^3\ch_0(A)<H^3.
\end{equation}
Combining (\ref{5.2}) and (\ref{5.3}), we conclude that
$\ch_0(A)=1$.

\medskip
\noindent{\bf Step 2.} {\em $A$ is a rank one sheaf, i.e.,
$\mathcal{H}^{-1}(A)=0$. }
\medskip

We argue by contradiction. Consider the long exact cohomology
sequence
\begin{equation}\label{exact}
0\rightarrow
\mathcal{H}^{-1}(A)\rightarrow\mathcal{O}_X\rightarrow\mathcal{H}^{-1}(F)
\rightarrow\mathcal{H}^{0}(A)\rightarrow\mathcal{O}_X(H)\rightarrow\mathcal{H}^{0}(F)\rightarrow0
\end{equation}
induced by $0\rightarrow A\rightarrow E_{\xi}\rightarrow
F\rightarrow0$. If $\mathcal{H}^{-1}(A)\neq0$, since
$\mathcal{H}^{-1}(F)$ is torsion free, one sees that
$\mathcal{H}^{-1}(A)=\mathcal{O}_X$ and $\mathcal{H}^{0}(A)$ is a
torsion free sheaf with
\begin{itemize}
\item $\rank\mathcal{H}^{0}(A)=2$;
\item $\ch_i(\mathcal{H}^{0}(A))=\ch_i(A)$ for $i\geq1$.
\end{itemize}

If $\mathcal{H}^{0}(A)$ is $\mu_H$-semistable, Bogomolov's
inequality gives
$$(H^2\ch_1(A))^2=\Big(H^2\ch_1(\mathcal{H}^{0}(A))\Big)^2\geq4H^3\cdot H\ch_2(\mathcal{H}^{0}(A))=4H^3\cdot H\ch_2(A).$$
This and (\ref{5}) imply that $H^2\ch_1(A)\geq2H^3$. But from
(\ref{5.3}), one obtains that $H^2\ch_1(A)<\frac{3}{2}H^3$. It is a
contradiction.

If $\mathcal{H}^{0}(A)$ is not $\mu_H$-semistable, we consider its
Harder-Narasimhan filtration:
$$0\subset M\subset\mathcal{H}^{0}(A)$$ and write
$Q:=\mathcal{H}^{0}(A)/M$. It turns out that $M$ and $Q$ are rank
one torsion free sheaves. By the definition of
$\Coh^{\frac{1}{2}}(X)$, one sees that
$\mu_H^-(\mathcal{H}^{0}(A))>\frac{1}{2}$. Hence
$\mu_H(Q)>\frac{1}{2}$. From the definition of HN-filtration and
$\Coh^{\frac{1}{2}}(X)$, it follows that
\begin{equation*}
H^2\ch_1(M)>\frac{1}{2}H^2\ch_1(A)>H^2\ch_1(Q)>\frac{H^3}{2}.
\end{equation*}
By (\ref{5.3}), we have
$$H^2\ch_1(M)+H^2\ch_1(Q)=H^2\ch_1(A)<\frac{3}{2}H^3.$$ Thus one
deduces
\begin{eqnarray}\label{5.5}
H^3>H^2\ch_1(M)>H^2\ch_1(Q)>\frac{H^3}{2}
\end{eqnarray}
On the other hand, Bogomolov's inequality and (\ref{5}) give
\begin{eqnarray*}
\nonumber\frac{1}{2}H^2\ch_1(M)+\frac{1}{2}H^2\ch_1(Q)&=&H\ch_2(M)+H\ch_2(Q)\\
&\leq&\frac{(H^2\ch_1(M))^2}{2H^3}+\frac{(H^2\ch_1(Q))^2}{2H^3}.
\end{eqnarray*}
One sees either
$\frac{1}{2}H^2\ch_1(M)\leq\frac{(H^2\ch_1(M))^2}{2H^3}$ or
$\frac{1}{2}H^2\ch_1(Q)\leq\frac{(H^2\ch_1(Q))^2}{2H^3}$. It follows
that either $H^2\ch_1(M)\geq H^3$ or $H^2\ch_1(Q)\geq H^3$. This
contradicts (\ref{5.5}).

To sum up, we have $\mathcal{H}^{-1}(A)=0$ and $A$ is rank one
sheaf.

\medskip
\noindent{\bf Step 3.} {\em $A$ is of the form $I_Z(H)$ for some
zero-dimensional subscheme $Z\subset X$. }
\medskip

We first show that $A$ is torsion free. If $A$ is not torsion free,
we denote by $A_t$ (resp. $A_{tf}:=A/A_t$) its torsion part (resp.
torsion free part). Since $A_t$ is a subobject of $E_{\xi}$ in
$\Coh^{\frac{1}{2}}(X)$, by the
$\nu_{\frac{1}{2},\frac{1}{2}}$-stability of $E_{\xi}$, one has
$$\nu_{\frac{1}{2},\frac{1}{2}}(A_t)=\frac{H\ch_2^{\frac{1}{2}}(A_t)}{H^2\ch_1^{\frac{1}{2}}(A_t)}\leq0.$$
This implies
\begin{equation}
H^2\ch_1(A_t)>0~
\mbox{and}~H\ch_2(A_t)-\frac{1}{2}H^2\ch_1(A_t)\leq0.
\end{equation}
By (\ref{5}), we obtain
$$H\ch_2(A_{tf})-\frac{1}{2}H^2\ch_1(A_{tf})\geq0.$$ Bogomolov's
inequality gives
\begin{equation}\label{5.8}
H^2\ch_1(A_{tf})\leq2H\ch_2(A_{tf})\leq\frac{(H^2\ch_1(A_{tf}))^2}{H^3}.
\end{equation}
Hence
\begin{equation}\label{5.9}
H^2\ch_1(A_{tf})\geq H^3.
\end{equation}

On the other hand, since $\mathcal{H}^{-1}(A)=0$ and $\rank
\mathcal{H}^{0}(A)=1$, the long exact sequence (\ref{exact}) gives
two short exact sequences
$$0\rightarrow\mathcal{O}_X\rightarrow\mathcal{H}^{-1}(F)\rightarrow
A_t\rightarrow0$$ and $$0\rightarrow
A_{tf}\rightarrow\mathcal{O}_X(H)\rightarrow\mathcal{H}^{0}(F)\rightarrow0.$$
Hence $A_{tf}$ is a subsheaf of $\mathcal{O}_X(H)$. By (\ref{5.9}),
one sees that $H^2\ch_1(A_{tf})=H^3$ and $\ch_1(A_{tf})=H$. This
shows that the chain inequalities of (\ref{5.8}) must be equalities.
This shows that
\begin{equation}\label{5.10}
H\ch_2(A_t)-\frac{1}{2}H^2\ch_1(A_t)=H\ch_2(A_{tf})-\frac{1}{2}H^2\ch_1(A_{tf})=0.
\end{equation}
Since $\ch_i(A_t)=\ch_i(\mathcal{H}^{-1}(F))$ for $i\geq1$ and
$\mathcal{H}^{-1}(F)$ is torsion free, similar as (\ref{5.8}) and
(\ref{5.9}), one obtains $$H^2\ch_1(\mathcal{H}^{-1}(F))\geq H^3,$$
in contradiction to
$$H^2\ch^{\frac{1}{2}}_1(\mathcal{H}^{-1}(F))=H^2\ch_1(\mathcal{H}^{-1}(F))-\frac{1}{2}H^3\leq0.$$
Therefore we conclude that $A_t=0$.


Then the equalities $\ch_1(A)=\ch_1(A_{tf})=H$ and (\ref{5}) imply
that $A=I_Z(H)$ for some zero-dimensional subscheme $Z\subset X$.
\end{proof}

\begin{proof}(Theorem \ref{Kodaira})
Since $A=I_Z(H)$, one sees that $\mathcal{H}^{-1}(F)=\mathcal{O}_X$
and $\mathcal{H}^{0}(F)=\mathcal{O}_Z$. We obtain an exact sequence
in $\Coh^{\frac{1}{2}}(X)$:
$$0\rightarrow\mathcal{O}_X[1]\rightarrow
F\rightarrow\mathcal{O}_Z\rightarrow0.$$ But
$$\ext^1(\mathcal{O}_Z, \mathcal{O}_X[1])=\hom(\mathcal{O}_X,\mathcal{O}_Z[1])=h^1(\mathcal{O}_Z)=0.$$
This implies $F\cong\mathcal{O}_X[1]\oplus\mathcal{O}_Z$. Hence we
obtain a surjective morphism
$$E_{\xi}\xrightarrow{g}\mathcal{O}_X[1]$$ in
$\Coh^{\frac{1}{2}}(X)$. From $$\Hom(\mathcal{O}_X[1],
I_Z(H))=\Hom(\mathcal{O}_X[1],\mathcal{O}_Z)=0,$$ it follows that
the composition morphism
$$\mathcal{O}_X[1]\xrightarrow{f}E_{\xi}\xrightarrow{g}\mathcal{O}_X[1]
$$ is nontrivial. Thus it is an isomorphism. This implies that the
exact sequence (\ref{5.1}) splits, so that $\xi=0$. This completes
the proof.
\end{proof}

\begin{remark}
Our proof of Proposition \ref{prop} is slightly different from that
of \cite[Proposition 3.3]{BBMT}. We do not use the dualizing functor
here. The reason is that in our situation
$h^1(X,\mathcal{O}_X(K_X+H))=1$ may not hold. Thus the self-duality
of $E_{\xi}$ (see \cite[Proposition 3.2]{BBMT}) cannot be obtained
directly.
\end{remark}

\bibliographystyle{amsplain}

\end{document}